\documentclass[a4paper9]{amsart}

\usepackage{amsmath}
\usepackage{amsthm}
\usepackage{amssymb}
\usepackage{cleveref}
\usepackage{graphicx}
\usepackage{tikz}
\usetikzlibrary{knots}
\usetikzlibrary{positioning}

\theoremstyle{plain}
\newtheorem{dfn}{Definition}
\newtheorem{thm}{Theorem}

\newtheorem{lem}{Lemma}

\newcommand{\C}{\mathbb{C}}
\newcommand{\Z}{\mathbb{Z}}
\newcommand{\sm}{\smallsetminus}

\newcommand{\iu}{\sqrt{-1}}
\newcommand{\veps}{\varepsilon}
\newcommand{\qsl}{\tilde{U}_{q} \mathfrak{sl}_{2}}
\newcommand{\unrolledqsl}{\bar{U}^{H}_{q}\mathfrak{sl}_{2}}

\DeclareMathOperator{\CGP}{N}
\DeclareMathOperator{\Op}{F}
\DeclareMathOperator{\modOp}{F^{\prime}}
\DeclareMathOperator{\qdim}{qdim}
\DeclareMathOperator{\modDim}{d}

\DeclareMathOperator{\id}{id}

\DeclareMathOperator{\CJ}{J}
\DeclareMathOperator{\ADO}{ADO}

\begin{document}

\title[Reshetikhin--Turaev and re-normalized link invariants]{Relations between Reshetikhin--Turaev and re-normalized link invariants}
\author[A. Mori]{Akihito Mori}
\address{Graduate School of Science, Tohoku University, Aramaki-aza-Aoba 6-3, Aoba-ku, Sendai, 980-8578, Japan}
\email{akihito.mori.s5@dc.tohoku.ac.jp}
\thanks{The first author is supported by JSPS KAKENHI Grant Number JP 21J10271.}


\begin{abstract}
    Costantino--Geer--Patureau-Mirand proved relations between the Reshetikhin--Turaev link invariants and the re-normalized link invariants for knots.
    Their theorem says that residues of the re-normalized link invariants are given by the Reshetikhin--Turaev link invariants.
    However, in the case of links, the residues are vanish, so we can not obtain any relations from the residues.
    In this paper, we prove that the Reshetikhin--Turaev link invariants appear in higher order terms of the re-normalized link invariants for links which are plumbed.
\end{abstract}


\maketitle
\tableofcontents

\section{Introduction} \label{sec: intro}

Reshetikhin--Turaev~\cite{WRTinv} constructed quantum invariants of 3-manifolds by using representations of the small quantum group $\qsl$ at roots of unity.
They reduces 3-manifolds to framed links in $S^{3}$ by surgery presentations.
For a link colored by symmetric tensor representations, they constructed the Reshetikhin--Turaev link invariants $\Op$ and took a weighted sum to obtain the Witten--Reshetikhin--Turaev (WRT) invariant.

There is a slightly generalized version of $\qsl$.
We call this quantum group the unrolled quantum group $\unrolledqsl$.
This quantum group has typical representations parametrized by complex numbers besides the representations of the small quantum group.
However $\Op(L)$ vanishes for any link $L$ which has a component colored by a typical representation because the representation has zero quantum dimension.
Geer--Patureau-Mirand--Turaev~\cite{modifiedqdim} introduced a modified quantum dimension $\modDim$ and constructed a new isotopy invariant $\modOp$.
We call the invariant the re-normalized Reshetikhin--Turaev link invariant.
Costantino--Geer--Patureau-Mirand~\cite{CGPinv} constructed nonsemisimple 3-manifold invariant (CGP invariant) $\CGP$ by a weighted sum of the re-normalized Reshetikhin--Turaev link invariant.

Costantino--Geer--Patureau-Mirand~\cite{relation} proved relations between $\Op$ and $\modOp$ for knots.
They showed that $\modOp$ of a knot in $S^{3}$ is a meromorphic function of its color and the residues are given by $\Op$.
Using the relations, Costantino--Geer--Patureau-Mirand~\cite{relation} and Costantino--Gukov--Putrov~\cite{WRTandCGP} proved relations between WRT invariant and CGP invariant of 3-manifolds presented by 0-framed knot.

One obstruction to generalizing their theorem stems from the fact that $\modOp$ of a link is a holomorphic function of its color, that is the residues vanish.
This is why the theorem is restricted to knots.

In this paper, we generalize their theorem for plumbed graphs.
The following is the main theorem of this paper.

\begin{thm} \label{thm: main}
    Fix a $2r$-th root of unity $q = \exp(\pi \iu / r)$.
    Let $\Gamma_{f} = (V, E, f)$ be a plumbed graph, where $f \colon V \to \Z$ is a framing.
    Take $x \in X_{r}^{V}$, then
    \begin{align*}
        \Op_{r}\left( \Gamma_{f}(r-1-x) \right)
        =
        &\frac{(-1)^{|E| + \sum_{v \in V_{\geq 2}} (d_{v}-1) x_{v}}}{r \{ 1 \}} \\
        &\times \sum_{\veps \in \{ \pm1 \}^{E}} 
        \left( \prod_{e \in E} \veps(e) \right) \lim_{\alpha \to x(\veps)} \frac{\modOp_{r}\left( \Gamma_{f}(\alpha) \right)}{\prod_{v \in V_{\geq 2}} \{ r\alpha_{v} \}^{d_{v}-1}}
    \end{align*}
    , where $d_{v}$ is a degree of a vertex $v \in V$, $X_{r}^{V} := \Z \sm r\Z$ and $\Gamma_{f}(\alpha)$ is a plumbed graph whose vertices are colored by $\alpha = (\alpha_{V})_{v \in V} \in \C^{V}$.
\end{thm}

In our case $\modOp$ is equal to the Akutsu--Deguchi--Ohtsuki invariant (see ~\cite{modifiedqdim}), so we can rephrase this theorem in terms of the colored Jones polynomial $\CJ$ and $\ADO_{r}$ as follows.

\begin{thm} \label{thm: colored Jones and ADO}
    There exist relations between the colored Jones polynomials and Akutsu--Deguchi--Ohtsuki invariants of plumbed graphs.
    \begin{align*}
        \CJ(\left( \Gamma_{f}(r-1-x) \right))|_{q = \exp(\pi \iu / r)}
        =
        &\frac{(-1)^{|E| + \sum_{v \in V_{\geq 2}} (d_{v}-1) x_{v}}}{r \{ 1 \}} \\
        &\times \sum_{\veps \in \{ \pm1 \}^{E}} 
        \left( \prod_{e \in E} \veps(e) \right) \lim_{\alpha \to x(\veps)} \frac{\ADO_{r}\left( \Gamma_{f}(\alpha) \right)}{\prod_{v \in V_{\geq 2}} \{ r\alpha_{v} \}^{d_{v}-1}}.
    \end{align*}
\end{thm}

The interesting aspect of this theorem is that the Reshetikhin--Turaev link invariants appear in the higher order terms of the re-normalized link invariants.
The proof of this theorem depends on a property of plumbed graphs, so it is not clear whether it can be generalized for other links or not.

Besides the above results, there are various relations of invariants constructed from non-integer weight representation of $\unrolledqsl$.
Murakami~\cite{Murakamia} constructed the colored Alexander invariant (which is a re-construction of the ADO invariant) and compared it with the colored Jones polynomial for non-plumbed links.
Murakami--Nagatomo also constructed the logarithmic invariant~\cite{MurakamiNagatomo} and proved relations between this invariant and the colored Jones, Alexander invariants.
After this paper, Murakami~\cite{MurakamiJones} expressed the logarithmic invariants of knots in terms of derivatives of the colored Jones invariants.
He also conjectured relations between the logarithmic invariants ant the hyperbolic volumes of the cone manifolds along knots and proved for the figure eight knot.
There is a conjecture about relations between GPPV invariants and ADO invariants~\cite{WRTandCGP}.
Beliakova--Hikami~\cite{BelHika} proved relations between Habiro's cyclotomic expansion of the colored Jones polynomial and the ADO invariants of the double twist knots.

For 3-manifold invariants, Costantino--Geer--Putrov~\cite{WRTandCGP} proved relations between CGP invariants and GPPV invariants under technical conditions.
Chen--Kuppun--Srinivasan~\cite{Chen} proved relations between Hennings invariants and WRT invariants.
Beliakova--Hikami~\cite{BelHika} proved relations between WRT invariants and CGP invariants of 3-manifolds obtained by 0-surgery on the double twist knots.
\section*{Acknowledgement} \label{sec:acknowledgement}

The work is supported by JSPS KAKENHI Grant Number JP 21J10271 and a Scholarship of Tohoku University, Division for Interdisciplinary Advanced Research and Education.
The author would like to show his greatest appreciation to Professor Yuji Terashima and Professor Jun Murakami for giving many pieces of advice. 
The author thanks his family for all the support.

\section{Preliminaries} \label{sec: pre}

\subsection{Notation} \label{subsec: notation}

All plumbed graphs in this paper are weighted trees.
Let $\Gamma_{f} = (V, E, f)$ stand for a plumbed graph where $V$ is a set of vertices, $E$ is a set of edges, and $f \colon V \to \Z$ is a weight.
For $v \in V$, let $d_{v}$ be the degree of $v$.
For $n \in \Z$, let $V_{\geq n}$ be a subset of $V$ consisting of vertices whose degrees are greater than $n$.

We identify a plumbed graph with a link.
Vertices correspond to unknots.
Weights correspond to framings.
If two vertices are connected by an edge, then corresponding unknots are linked as a Hopf link.
For example, see \Cref{fig: graph}.
In this paper, we let $H$ be a Hopf link.

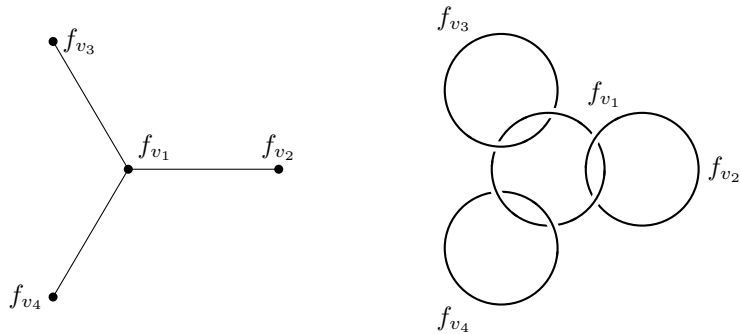
\begin{figure}[htb]
	\begin{minipage}[htb]{0.45\linewidth}
		\centering
		\begin{tikzpicture}
            \coordinate[label=45:$f_{v_{1}}$](O)at(0, 0);
            \coordinate[label=above:$f_{v_{2}}$](A)at(2, 0);
            \coordinate[label=right:$f_{v_{3}}$](B)at(-1, 1.7);
            \coordinate[label=left:$f_{v_{4}}$](C)at(-1, -1.7);
            \foreach \P in{O, A, B, C}
            {
                \fill[black](\P)circle(0.06);
            }

            \draw (O) -- (A);
            \draw (O) -- (B);
            \draw (O) -- (C);
		\end{tikzpicture}
	\end{minipage}
	\begin{minipage}[htb]{0.45\linewidth}
		\centering
		\begin{tikzpicture}[scale=0.5]
			\begin{knot}[
				clip width=5, 
				flip crossing=1, 
				flip crossing=4, 
				flip crossing=6
				]
				\strand[thick] (0, 0) circle [radius=1.5];
				\strand[thick] (2.5, 0) circle [radius=1.5];
				\strand[thick] (-1.25, 2.1) circle [radius=1.5cm];
				\strand[thick] (-1.25, -2.1) circle [radius=1.5cm];

				\node (1) at (1.5, 2) {$f_{v_{1}}$};
				\node (2) at (4.7, 0) {$f_{v_{2}}$};
				\node (3) at (-2.5, 4) {$f_{v_{3}}$};
				\node (4) at (-2.5, -4) {$f_{v_{4}}$};
			\end{knot}
		\end{tikzpicture}
	\end{minipage}
\caption{A plumbed graph and the corresponding framed link}
\label{fig: graph}
\end{figure}

Finally, we define a symbol used in the proof of the main theorem.
Fix a vertex $v_{0}$.
For any vertex $v$, there is a unique path $P_{v}$ from $v_{0}$ to $v$.
\begin{dfn}
    For $x \in \Z^{V}$ and $\veps \in \{ \pm1 \}^{E}$, we define $x(\veps) \in \Z^{V}$ by 
    \[
        x_{v}(\veps) = \left( \prod_{e \in P_{v}} \veps_{e} \right) x_{v}.
    \]
\end{dfn}

\subsection{Quantum groups and its representations} \label{subsec: quantum group}

Fix a positive integer $r$ and let $q = \exp(\pi \iu / r)$ be a $2r$ th-root of unity.
For a complex number $z$, we set 
\begin{align*}
    &q^{z} = \exp \left( \frac{\pi \iu z}{r} \right), & &\{ z \} = q^{z} - q^{-z}, & &[z] = \frac{\{ z \}}{\{ 1 \}}.
\end{align*}

We recall two quantum groups.
Consider a $\C$-algebra $\tilde{U}_{q} \mathfrak{sl}_{2}$ generated by $E, F, K^{\pm1}$ whose relations are
\begin{align*}
    &KK^{-1} = 1 = K^{-1}K, & &KEK^{-1} = q^{2}E, & &KFK^{-1} = q^{-2}F, \\
    &[E, F] = \frac{K - K^{-1}}{q - q^{-1}}, & &E^{r} = F^{r} = 0, & &K^{2r} = 1.
\end{align*}
We define the coproduct $\Delta \colon \tilde{U}_{q} \mathfrak{sl}_{2} \to \tilde{U}_{q} \mathfrak{sl}_{2} \otimes \tilde{U}_{q} \mathfrak{sl}_{2}$, counit $\epsilon \colon \tilde{U}_{q} \mathfrak{sl}_{2} \to \C$ and antipode $S \colon \tilde{U}_{q} \mathfrak{sl}_{2} \to \tilde{U}_{q} \mathfrak{sl}_{2}$ as follows:
\begin{align*}
    &\Delta(E) = 1 \otimes E + E \otimes K, & &\Delta(F) = K^{-1} \otimes F + F \otimes 1, & &\Delta(K) = K \otimes K, \\
    &\epsilon(E) = 0, & &\epsilon(F) = 0, & &\epsilon(K) = 1, \\
    &S(E) = -EK^{-1}, & &S(F) = -KF, & &S(K) = K^{-1}.
\end{align*}
Combined with these maps, $\tilde{U}_{q} \mathfrak{sl}_{2}$ is a Hopf algebra.
We call $\tilde{U}_{q} \mathfrak{sl}_{2}$ the small quantum group.

For each $n \in \{ 0, \dots, r-1 \}$, there is an irreducible highest weight representation $S_{n}$.
This module has a basis $\{ s_{i} \}_{i=0}^{n}$ and actions are given by
\begin{align*}
    & Es_{i} = [i][n-i+1]s_{i-1}, & &Fs_{i} = s_{i+1}, & &Ks_{i} = q^{n-2i}s_{i}, & &Es_{0} = 0 = Fs_{n}.
\end{align*}

Another quantum group $U^{H}_{q} \mathfrak{sl}_{2}$ is a $\C$-algebra generated by $E, F, H, K^{\pm1}$ and its relations are given by
\begin{align*}
    &KK^{-1} = 1 = K^{-1}K, & &KEK^{-1} = q^{2}E, & &KFK^{-1} = q^{-2}F, \\
    &[E, F] = \frac{K - K^{-1}}{q - q^{-1}}, & &E^{r} = F^{r} = 0, \\
    &[H, K] = 0, & &[H, E] = 2E, & &[H, F] = -2F.
\end{align*}
We extend the coproduct, counit and antipode of the small quantum group by
\begin{align*}
    &\Delta(H) = H \otimes 1 + 1 \otimes H, & &\epsilon(H) = 0, & &S(H) = -H.
\end{align*}
We call this Hopf algebra $U^{H}_{q} \mathfrak{sl}_{2}$ the unrolled quantum group.
A $U^{H}_{q} \mathfrak{sl}_{2}$-module $V$ is called a weight module if $V$ splits into a direct sum of eigenspaces of $H$ and $K$ operates on $V$ as $q^{H}$.
Since we do not require $K^{2r} = 1$, eigenvalues of $H$ can be complex numbers.
For each $\alpha \in \C$, there is a highest weight representation $V_{\alpha}$.
The module $V_{\alpha}$ has a basis $\{ v_{i} \}_{i=0}^{r-1}$ and actions are given by
\begin{align*}
    &Ev_{i} = [i][i-\alpha]v_{i-1}, & &Fv_{i} = v_{i+1}, & &Hv_{i} = (\alpha + r - 1 - 2i)v_{i}, & &Ev_{0} = 0 = Fv_{r-1}.
\end{align*}
We call $V_{\alpha}$ typical if $\alpha \in \C \sm X_{r}$ and an atypical if $\alpha \in X_{r}$.
\subsection{Invariants} \label{sec: invariants}

Let $\Op_{r}$ be the Reshetikhin--Turaev functor and $M$ be $S_{n}$ or a typical module $V_{\alpha}$.
If $L$ is a framed link with an edge colored by $M$ and $T_{L}$ is a ribbon graph obtained by cutting the edge of $K$ then
\[
    \Op_{r}(L) = \qdim(M) \langle T_{L} \rangle
\]
, where $\langle T_{L} \rangle$ is a scalar satisfying $\Op_{r}(T_{L}) = \langle T_{L} \rangle \id_{M}$.
Since $\qdim(V_{\alpha}) = 0$, $\Op_{r}(L) = 0$ for links with edges colored by $V_{\alpha}$.

Geer--Patureau-Mirand--Turaev~\cite{modifiedqdim} introduced a modified quantum dimension $\modDim \colon \C \sm X_{r} \to \C$ defined by
\[
    \modDim(\alpha) = (-1)^{r-1} r \frac{\{ \alpha \}}{\{ r\alpha \}}.
\]
They constructed the re-normalized Reshetikhin--Turaev link invariant $\modOp_{r}$ by replacing the quantum dimension $\qdim$ with the modified quantum dimension $\modDim$.

In the rest of this subsection, we describe properties of link invariants $\Op_{r}$ and $\modOp_{r}$.
The values of the invariants for a Hopf link $H$ and a twisted edge are well-known.

\begin{lem} \label{lem: hopf}
    \begin{enumerate}
        \item Take $\alpha, \beta \in \C \sm X_{r}$, then
        \[ \modOp_{r}(H(\alpha, \beta)) = (-1)^{r-1} r q^{\alpha \beta}. \]
        \item Take $0 \leq m, n \leq r-1$, then
        \[ \Op_{r}(H(m, n)) = (-1)^{m+n} \frac{\{ (m+1)(n+1) \}}{\{ 1 \}}. \]
    \end{enumerate}
\end{lem}

\begin{lem} \label{lem: twist}
    \begin{enumerate}
        \item Let $e$ be an edge colored by $V_{\alpha}$ and have $+1$ framing.
        Then 
        \[ \modOp_{r}(e) = q^{(\alpha^{2} - (r-1)^{2})/2} \id_{V_{\alpha}}. \]
        \item Let $e$ be an edge colored by $S_{x}$ and have $+1$ framing.
        Then
        \[ \Op_{r}(e) = (-1)^{n}q^{(n^{2} + 2n)/2} \id_{S_{x}}. \]
    \end{enumerate}
\end{lem}

The invariants are well-behaved under a connected sum.
\begin{lem} \label{lem: connected_sum}
    \begin{enumerate}
        \item Take $\alpha \in \C \sm X_{r}$ and let $T, T'$ be ribbon graphs with edges $e, e'$ colored by $V_{\alpha}$.
        For a connected sum $T \# T'$ along the edges, $\modOp_{r}$ satisfies
        \[ \modOp_{r}(T \# T') = \modDim(\alpha)^{-1} \modOp_{r}(T) \modOp_{r}(T'). \]
        \item Take $x \in \{ 0, \dots, r-1 \}$ and let $T, T'$ be ribbon graphs with edges $e, e'$ colored by $S_{x}$.
        For a connected sum $T \# T'$ along the edges, $\Op_{r}$ satisfies
        \[ \Op_{r}(T \# T') = \qdim(S_{x})^{-1} \Op_{r}(T) \Op_{r}(T'). \]
    \end{enumerate}
\end{lem}

The following lemmas are technical lemmas appeard in the proof of the main theorem.
\begin{lem} \label{lem: limit sign}
    Let $H_{(0, f_{v'})}(\alpha_{v_{1}}, \alpha_{v'})$ be a Hopf link colored by $(\alpha_{v_{1}}, \alpha_{v'}) \in (\C \sm X_{r})^{2}$ with framing $(0, f_{v'})$.
    Let $\Gamma$ be a plumbed graph and $\Gamma'$ be a connected sum $\Gamma \# H$ at $v_{1}$.
    For $\veps_{H} \in \{ \pm1 \}$ and $\veps' \in \{ \pm1 \}^{V'}$,
    \[
        \lim_{\substack{\alpha_{v_{1}} \to x_{v_{1}} \\ \alpha_{v'} \to \veps_{H}(e') x_{v'}}}
        \modOp_{r}(H_{(0, f_{v'})}(\alpha_{v_{1}}, \alpha_{v'}))
        =
        \lim_{\substack{\alpha_{v_{1}} \to x_{v_{1}}(\veps') \\ \alpha_{v'} \to x_{v'}(\veps')}}
        \modOp_{r}(H_{(0, f_{v'})}(\alpha_{v_{1}}, \alpha_{v'}))
    \]
    , where $V'$ is a set of vertices of $\Gamma'$.
\end{lem}

\begin{proof}
    \begin{align*}
        \lim_{\substack{\alpha_{v_{1}} \to x_{v_{1}} \\ \alpha_{v'} \to \veps_{H}(e') x_{v'}}}
        \modOp_{r}(H_{(0, f_{v'})}(\alpha_{v_{1}}, \alpha_{v'}))
        &=
        \lim_{\substack{\alpha_{v_{1}} \to x_{v_{1}} \\ \alpha_{v'} \to \veps_{H}(e') x_{v'}}}
        (-1)^{r-1} r q^{\alpha_{v_{1}} \alpha_{v'}} \\
        &=
        \lim_{\substack{\alpha_{v_{1}} \to x_{v_{1}} \\ \alpha_{v'} \to \veps_{H}(e') x_{v'}}}
        (-1)^{r-1} r q^{\left( \prod_{e \in P_{v_{1}}} \veps(e) \alpha_{v_{1}} \right) \left( \prod_{e \in P_{v_{1}}} \veps(e) \alpha_{v'} \right)}.
    \end{align*}
    Let $\alpha'_{v_{1}} = \prod_{e \in P_{v_{1}}} \veps(e) \alpha_{v_{1}}$ and $\alpha'_{v'} = \prod_{e \in P_{v_{1}}} \veps(e) \alpha_{v'}$, then we obtain the right hand side.
\end{proof}

\begin{lem} \label{lem: limit}
    Take $\alpha \in (\C \sm X_{r})^{V}, \alpha_{v'} \in \C \sm X_{r}, x \in X_{r}^{V}, x_{v'} \in X_{r}$.
    For $\alpha' = (\alpha, \alpha_{v'}) \in (\C \sm X_{r})^{V'}, x' = (x, x_{v'}) \in (\C \sm X_{r})^{V}$, $\modOp_{r}$ satisfies
    \[
        \lim_{\alpha' \to x'} 
        \frac{
            \modOp_{r}(\Gamma \# H_{(0, f')}(\alpha_{v_{1}}, \alpha_{v'}))
        }{
                \prod_{v \in V'_{\geq 2}} \{ r\alpha_{v} \}^{d'_{v}-1}
        }
        =
        q^{f'(x_{v'}^{2} - (r-1)^{2})/2}
        \frac{
                q^{x_{v_{1}} x_{v'}}
        }{
            \{ x_{v_{1}} \}
        } 
        \lim_{\alpha \to x} 
        \frac{
            \modOp_{r}(\Gamma)
        }{
            \prod_{v \in V_{\geq 2}} \{ r\alpha_{v} \}^{d_{v}-1}.
        }
    \]
\end{lem}

\begin{proof}
    The connected sum splits by \Cref{lem: connected_sum} as below:
    \begin{align*}
        \frac{\modOp_{r}(\Gamma_{f}(\alpha) \# H_{(0, f')}(\alpha_{v_{1}}, \alpha_{v'}))}{\prod_{v \in V'_{\geq 2}} \{ r\alpha_{v} \}^{d'_{v}-1}}
        &=
        \frac{
            \modDim(\alpha_{v_{1}})^{-1} \modOp_{r}(\Gamma_{f}(\alpha)) \modOp_{r}(H_{(0, f')}(\alpha_{v_{1}}, \alpha_{v'}))
        }{
            \prod_{v \in V'_{\geq 2}} \{ r\alpha_{v} \}^{d'_{v}-1}
        }\\
        &=
        q^{f'(\alpha_{v'}^{2} - (r-1)^{2})/2}
        \frac{
            q^{\alpha_{v_{1}} \alpha_{v'}}
        }{
            \{ \alpha_{v_{1}} \}
        }
        \frac{
            \{ r\alpha_{v_{1}}  \} \modOp_{r}(\Gamma_{f}(\alpha))
        }{
            \prod_{v \in V'_{\geq 2}} \{ r\alpha_{v} \}^{d'_{v}-1}.
        }
    \end{align*}
    Since $d'_{v_{1}} = d_{v_{1}} + 1$, we obtain
    \[
        q^{f'(\alpha_{v'}^{2} - (r-1)^{2})/2}
        \frac{
            q^{\alpha_{v_{1}} \alpha_{v'}}
        }{
            \{ \alpha_{v_{1}} \}
        }
        \frac{
            \{ r\alpha_{v_{1}}  \} \modOp_{r}(\Gamma_{f}(\alpha))
        }{
            \prod_{v \in V'_{\geq 2}} \{ r\alpha_{v} \}^{d'_{v}-1}
        }
        =
        q^{f'(\alpha_{v'}^{2} - (r-1)^{2})/2}
        \frac{
            q^{\alpha_{v_{1}} \alpha_{v'}}
        }{
            \{ \alpha_{v_{1}} \}
        }
        \frac{
            \modOp_{r}(\Gamma_{f}(\alpha))
        }{
            \prod_{v \in V_{\geq 2}} \{ r\alpha_{v} \}^{d_{v}-1}
        }.
    \]
    By the above equation, 
    \[
        \frac{\modOp_{r}(\Gamma'_{f'}(\alpha'))}{\prod_{v \in V'_{\geq 2}} \{ r\alpha_{v} \}^{d'_{v}-1}}
        =
        q^{f'(\alpha_{v'}^{2} - (r-1)^{2})/2}
        \frac{
            q^{\alpha_{v_{1}} \alpha_{v'}}
        }{
            \{ \alpha_{v_{1}} \}
        }
        \frac{
            \modOp_{r}(\Gamma_{f}(\alpha))
        }{
            \prod_{v \in V_{\geq 2}} \{ r\alpha_{v} \}^{d_{v}-1}
        }
    \]
    holds.
    Taking a limit, we completes the proof.
\end{proof}

\section{A proof of the main Theorem} \label{sec: main}

We prove \Cref{thm: main} by induction.
When $\Gamma$ is a Hopf link $H$ (see \Cref{fig: hopf}), the left hand side reduces to
\begin{align*}
    \Op_{r}(H_{f}(r-1-x)) 
    &= 
    \left( \prod_{v = v_{1}, v'} (-1)^{f_{v}(r-1-x_{v})} q^{f_{v}(x_{v}^{2} + 2x_{v})/2} \right)
    (-1)^{x_{v_{1}} + x_{v'}} \frac{\{ (r - x_{v_{1}})(r - x_{v'}) \}}{\{ 1 \}} \\
    &=
    \left( \prod_{v = v_{1}, v'} (-1)^{f_{v}(r-1-x_{v})} q^{f_{v}(x_{v}^{2} + 2x_{v})/2} \right)
    \frac{(-1)^{1}}{r\{ 1 \}}
    (-1)^{r-1}r
    \{ x_{v_{1}}x_{v'} \}
\end{align*}
by \Cref{lem: hopf} and \Cref{lem: twist}.

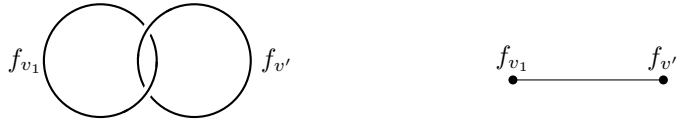
\begin{figure}[b]
	\begin{minipage}[htb]{0.45\linewidth}
		\centering
        \begin{tikzpicture}[scale=0.5]
			\begin{knot}[
				clip width=5, 
				flip crossing=1, 
				flip crossing=4, 
				flip crossing=6
				]
				\strand[thick] (0, 0) circle [radius=1.5];
				\strand[thick] (2.5, 0) circle [radius=1.5];

                \node (1) at (-2, 0) {$f_{v_{1}}$};
				\node (2) at (4.7, 0) {$f_{v'}$};
			\end{knot}
		\end{tikzpicture}
	\end{minipage}
	\begin{minipage}[htb]{0.45\linewidth}
		\centering
        \begin{tikzpicture}
            \coordinate[label=above:$f_{v_{1}}$](O)at(0, 0);
            \coordinate[label=above:$f_{v'}$](A)at(2, 0);
            \foreach \P in{O, A}
            {
                \fill[black](\P)circle(0.06);
            }

            \draw (O) -- (A);
		\end{tikzpicture}
	\end{minipage}
\caption{A Hopf link and the corresponding plumbed graph $H$}
\label{fig: hopf}
\end{figure}

We expand $\{ x_{v_{1}}x_{v'} \}$ and take $v_{1}$ as a base point to obtain the following equation: 
\[
    \frac{(-1)^{1}}{r\{ 1 \}}
    (-1)^{r-1}r
    \{ x_{v_{1}}x_{v'} \}
    =
    \frac{(-1)^{1}}{r\{ 1 \}}
    \sum_{\veps(e') \in \{ \pm1 \}} \veps(e')
    (-1)^{r-1}r q^{x_{v_{1}}(\veps)x_{v'}(\veps)}.
\]
We can prove 
\[
    (-1)^{f_{v}(r-1-x_{v})} q^{f_{v}(x_{v}^{2} + 2x_{v})/2} 
    =
    q^{f_{v}(x_{v}^{2} - (r-1)^{2})/2}
\]
by direct calculation.
Applying \Cref{lem: hopf}, we show that \Cref{thm: main} holds for $\Gamma = H$:
\[
    \Op_{r}(H_{f}(r-1-x)) 
    =
    \frac{(-1)^{1}}{r\{ 1 \}}
    \sum_{\veps(e') \in \{ \pm1 \}} \veps(e')
    \lim_{\alpha \to x} \modOp_{r}(H(\alpha)).
\]

Assume that the equation holds for a plumbed graph $\Gamma$ and we will show that \Cref{thm: main} still holds for a connected sum $\Gamma' = \Gamma \# H$ with a Hopf link $H$.
We fix a vertex $v_{1}$ where $\Gamma$ and $H$ are joined.
See \Cref{fig: connected sum}.

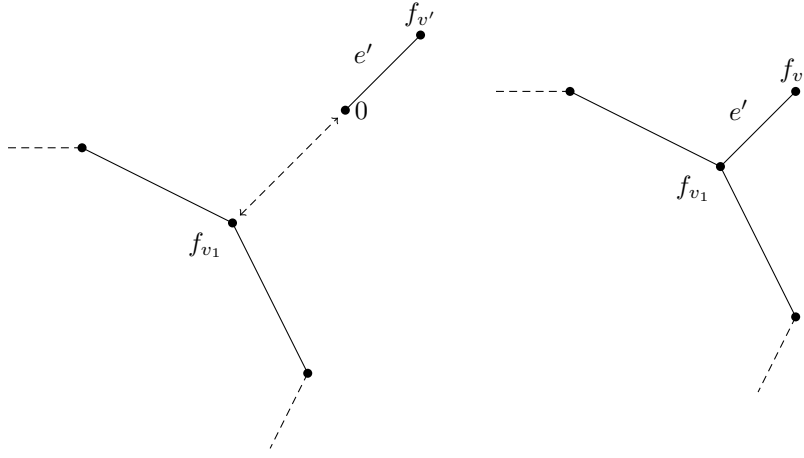
\begin{figure}[htb]
	\begin{minipage}[htb]{0.45\linewidth}
		\centering
		\begin{tikzpicture}
            \coordinate[label=225:$f_{v_{1}}$](O)at(0, 0);
            \coordinate(A)at(-2, 1);
            \coordinate(B)at(1, -2);

            \coordinate[label=right:$0$](C)at(1.5, 1.5);
            \coordinate[label=above:$f_{v'}$] (D) at (2.5, 2.5);
            \foreach \P in{O, A, B, C, D}
            {
                \fill[black](\P)circle(0.06);
            }

            \coordinate[label=135:$e'$] (E) at (2, 2);

            \draw (O) -- (A);
            \draw (O) -- (B);
            \draw[densely dashed] (A) -- (-3, 1);
            \draw[densely dashed] (B) -- (0.5, -3);

            \draw (C) -- (D);

            \draw[<->, densely dashed] (0.1, 0.1) -- (1.4, 1.4);
		\end{tikzpicture}
	\end{minipage}
	\begin{minipage}[htb]{0.45\linewidth}
		\centering
		\begin{tikzpicture}
            \coordinate[label=225:$f_{v_{1}}$](O)at(0, 0);
            \coordinate(A)at(-2, 1);
            \coordinate(B)at(1, -2);
            \coordinate[label=above:$f_{v'}$] (C) at (1, 1);

            \foreach \P in{O, A, B, C}
            {
                \fill[black](\P)circle(0.06);
            }

            \coordinate[label=135:$e'$] (E) at (0.5, 0.5);

            \draw (O) -- (A);
            \draw (O) -- (B);
            \draw[densely dashed] (A) -- (-3, 1);
            \draw[densely dashed] (B) -- (0.5, -3);

            \draw (O) -- (C);
		\end{tikzpicture}
	\end{minipage}
\caption{$\Gamma'$ obtained by joining $\Gamma$ and $H$ at $v_{1}$}
\label{fig: connected sum}
\end{figure}

Then \Cref{lem: connected_sum} shows
\begin{align*}
    &\Op_{r}\left( \Gamma_{f}(r-1-x) \# H_{(0, f_{v'})}(r-1-x_{v_{1}}, r-1-x_{v'}) \right) \\
    &=
    \qdim(S_{r-1-x_{v_{1}}})^{-1} 
    \Op_{r}\left( \Gamma_{f}(r-1-x) \right) 
    \Op_{r}\left( H_{(0, f_{v'})}(r-1-x_{v_{1}}, r-1-x_{v'}) \right).
\end{align*}
From the induction hypothesis, we find that the right hand side of the above equation equals to
\begin{align*}
    &(-1)^{r-1-x_{v_{1}}} \frac{\{1\}}{\{r - x_{v_{1}}\}} \\
    &\times
    \frac{
            (-1)^{|E| + \sum_{v \in V_{\geq 2}} (d_{v}-1) x_{v}}
        }{
            r \{ 1 \}
        }
    \sum_{\veps \in \{ \pm1 \}^{E}}
    \left( \prod_{e \in E} \veps(e) \right)
    \lim_{\alpha \to x(\veps)} 
    \frac{
            \modOp_{r}(\Gamma_{f}(\alpha))
        }{
            \prod_{v \in V_{\geq 2}} \{ r \alpha_{v} \}^{d_{v}-1}
        } \\
    &\times
    \frac{(-1)}{r\{1\}}
    \sum_{\veps_{H}(e') \in \{ \pm1 \}} \veps(e') 
    \lim_{\alpha_{v_{1}} \to x_{v_{1}}, \alpha_{v'} \to \veps_{H}(e') x_{v'}} \modOp_{r}(H_{0, f_{v'}}(\alpha_{v_{1}}, \alpha_{v'})).
\end{align*}
\Cref{lem: limit sign} shows 
\begin{align*}
    &\frac{
            (-1)^{|E'| + \sum_{v \in V'_{\geq 2}} (d'_{v}-1) x_{v}}
        }{
            r \{ 1 \}
        }
    \frac{(-1)^{r-1}}{\{x_{v_{1}}\}} \\
    &\times
    \sum_{\veps' \in \{\pm1\}^{E'}}
    \left( \prod_{e \in E'} \veps'(e) \right)
    \lim_{\alpha' \to x'(\veps')}
    \frac{
         \modOp_{r}(\Gamma_{f}(\alpha))
         \modOp_{r}(H_{(0, f_{v'})}(\alpha_{v_{1}}, \alpha_{v'}))
       }{
         \prod_{v \in V_{\geq 2}} \{ r \alpha_{v} \}^{d_{v}-1}
        }.
\end{align*}
From \Cref{lem: connected_sum}, we see
\[
    \modOp_{r}(\Gamma_{f}(\alpha))
    \modOp_{r}(H_{(0, f_{v'})}(\alpha_{v_{1}}, \alpha_{v'}))
    =
    (-1)^{r-1}r\frac{\alpha_{v_{1}}}{r\alpha_{v_{1}}}
    \modOp_{r}(\Gamma_{f}(\alpha) \# H_{(0, f_{v'})}(\alpha_{v_{1}}, \alpha_{v'})).
\]
Applying the equation, we obtain
\begin{align*}
    \Op_{r}\left( \Gamma'_{f'}(r-1-x) \right)
    =
    &\frac{
            (-1)^{|E'| + \sum_{v \in V'_{\geq 2}} (d'_{v}-1) x_{v}}
        }{
            r \{ 1 \}
        } \\
    &\sum_{\veps' \in \{\pm1\}^{E'}}
    \left( \prod_{e \in E'} \veps'(e) \right)
    \lim_{\alpha' \to x'(\veps')}
    \frac{
         \modOp_{r}(\Gamma'_{f'}(\alpha'))
       }{
         \prod_{v \in V'_{\geq 2}} \{ r \alpha_{v} \}^{d'_{v}-1}
        }.
\end{align*}
This completes the induction.


\bibliographystyle{alpha}
\bibliography{relations}

\end{document}